\documentclass{amsart}
\usepackage{amssymb,amsmath,amsthm,newlfont,enumerate}

\theoremstyle{plain}
\newtheorem{theorem}{Theorem}[section]
\newtheorem{lemma}[theorem]{Lemma}
\newtheorem{corollary}[theorem]{Corollary}
\newtheorem{proposition}[theorem]{Proposition}
\newtheorem{problem}[theorem]{Problem}

\theoremstyle{definition}

\theoremstyle{remark}
\newtheorem*{remark}{Remark}

\newcommand{\CC}{{\mathbb C}}
\newcommand{\DD}{{\mathbb D}}
\newcommand{\RR}{{\mathbb R}}
\newcommand{\TT}{{\mathbb T}}
\newcommand{\ZZ}{{\mathbb Z}}
\newcommand{\cD}{{\mathcal D}}
\newcommand{\cH}{{\mathcal H}}

\renewcommand{\hat}{\widehat}

\begin{document}

\title[Coefficient formula for de Branges--Rovnyak norms]{On the coefficient formula for\\ de Branges--Rovnyak norms}

\author[T. Ransford]{Thomas Ransford}

\address{%
D\'epartement de math\'ematiques et de statistique, Universit\'e Laval, Qu\'ebec (QC), G1V 0A6, Canada}

\email{ransford@mat.ulaval.ca}

\thanks{Research supported by  NSERC Discovery Grant RGPIN-2026-04565}



\date{4 September 2026}

\begin{abstract}
Let $\cH(b)$ be the de Branges--Rovnyak space associated
to a non-extreme point $b$ of the unit ball of $H^\infty$,
and let $\phi=b/a$, where $a$ is the Pythagorean mate of $b$.
It is known that, if $f$ is a function holomorphic on a neighbourhood
of the closed unit disk, then it belongs to $\cH(b)$,
and its norm in $\cH(b)$
can be expressed in terms of the Taylor coefficients of $f$ and $\phi$ via the formula
\[
\|f\|_{\cH(b)}^2=\sum_{m\ge0}|\hat{f}(m)|^2
+\sum_{m\ge0}\Bigl|\sum_{n\ge0}\overline{\hat{\phi}(n)}\hat{f}(m+n)\Bigr|^2.
\]
However,  the formula can break down for some other 
$f\in\cH(b)$.

In this article we extend the scope of the formula to all $f\in H^2$ 
for which the right-hand side is finite, provided that either 
$\phi\in H^2$ or $\phi$ is rational.
If merely $\phi\in H^p$ for some $p\in(0,2]$,
then the formula still holds provided that, in addition,
$\sum_{m\ge0}m^{2/p-1}|\hat{f}(m)|^2<\infty$.
We also establish a limit-form of the formula that is 
valid for all non-extreme $b$ and all $f\in\cH(b)$.
\end{abstract}

\maketitle

\section{Introduction}

The spaces now called de Branges--Rovnyak spaces 
were introduced by de Branges and Rovnyak 
in the appendix of \cite{dBR66a} 
and further studied in \cite{dBR66b}.
They were later popularized by the book of Sarason \cite{Sa94}.
It is now known that these spaces have a variety of connections 
with other topics in complex analysis and operator theory,
and they have attracted the interest of many authors.

When working with de Branges--Rovnyak spaces, 
it is very helpful 
to have a formula for the norm (and, via polarization,
for the inner product) of functions in these spaces,
expressed  in terms of their Taylor coefficients.
There does indeed exist such a formula, namely,
\begin{equation}\label{E:formula}
\|f\|_{\cH(b)}^2\text{~``=''~}\sum_{m\ge0}|\hat{f}(m)|^2
+\sum_{m\ge0}\Bigl|\sum_{n\ge0}\overline{\hat{\phi}(n)}\hat{f}(m+n)\Bigr|^2.
\end{equation}
Here $\cH(b)$ is the de Branges--Rovnyak space
associated to a non-extreme point $b$ of the unit ball of $H^\infty$,
and $\phi=b/a$, where $a$ is the Pythagorean mate of $b$
(all these terms will be explained in \S\ref{S:deBR} below).
Also $\hat{f}(n),\hat{\phi}(n)$  are the $n$-th
coefficients  of the Taylor expansions around $0$ of $f,\phi$ respectively.

What is noteworthy about the formula \eqref{E:formula}
is that it is only sometimes 
true---hence the inverted commas around the equals sign.
It \emph{is} always true if $f$ is a polynomial.
Indeed, the special case when $f(z)=z^k$, namely
\[
\|z^k\|_{\cH(b)}^2=1+\sum_{n=0}^k|\hat{\phi}(n)|^2,
\]
is well known and used in many places in the literature.
More generally, equality holds in \eqref{E:formula}
if $f$ is holomorphic in a neighbourhood of the closed unit disk;
this is implicit in \cite{Sa08}, and proved explicitly in 
\cite{CGR10}. As polynomials are dense
in $\cH(b)$, one might be tempted to infer that the same formula
holds for all $f\in\cH(b)$, but in fact this turns out to be false
in general.
Counterexamples can be found, for example, in \cite{EFKMR16}
and \cite{MR18}, and there are more in this article.

In view of the potential applications, 
it is of great interest to widen the scope of validity
of the formula to include as many functions $f$ as possible.
This is the goal of the present paper.
Here is a summary of our main results:
\begin{itemize}
\item If $\phi$ belongs to the Hardy space $H^2$, 
then equality holds in \eqref{E:formula}
for all $f\in\cH(b)$ (Theorem~\ref{T:H2}).
On the other hand, this is no longer necessarily the
case if merely $\phi\in \cap_{p<2}H^p$
(Corollary~\ref{C:p<2}).
\item If $\phi$ belongs to $H^p$ for some $p\in(0,2]$, then 
equality holds in \eqref{E:formula} for all $f\in H^2$
for which the right-hand side of \eqref{E:formula} is finite
and such that, in addition, $\sum_{m\ge0}m^{2/p-1}|\hat{f}(m)|^2<\infty$
(Theorem~\ref{T:Hp}).
\item If $\phi$ is rational 
then equality holds in \eqref{E:formula} for all $f\in H^2$
such that the right-hand side of \eqref{E:formula} is finite.
If, in addition,
all the poles of $\phi$ on the unit circle $\TT$ are simple,
then equality holds in \eqref{E:formula} for all $f\in \cH(b)$.
On the other hand, if $\phi$ has a multiple pole on $\TT$, 
then \eqref{E:formula} breaks down for at least one $f\in\cH(b)$
(Theorem~\ref{T:rational}).
\item There is a limit form of the formula \eqref{E:formula}
that is  valid for all Pythag\-orean pairs $(b,a)$  and all $f\in\cH(b)$ (Theorem~\ref{T:Tapprox}).
\end{itemize}

%%%%%%%%%%%%%%%%%%%%%%%%%%%%%%%%

\section{Background on de Branges--Rovnyak spaces}\label{S:deBR}

In this section we review  de Branges--Rovnyak spaces and some related notions
that will be needed in the sequel. The treatment is deliberately brief. 
The interested reader may find much more information in the books
of Sarason \cite{Sa94} and Fricain--Mashreghi \cite{FM16}.

\subsection{De Branges--Rovnyak spaces}
Let $b\in H^\infty$ with $\|b\|_{H^\infty}\le 1$. 
Then the map $(z,w)\mapsto (1-\overline{b(w)}b(z))/(1-\overline{w}z)$ is
positive-definite on $\DD\times\DD$, so it is the reproducing kernel 
of a unique Hilbert function space on $\DD$. This space is
called the \emph{de Branges--Rovnyak space} associated to $b$,
and is denoted by $\cH(b)$. It is always contractively contained
in the Hardy space $H^2$.

The theory of de Branges--Rovnyak spaces is governed by a fundamental
dichotomy, namely whether or not $b$ is  an extreme point of the unit ball of~$H^\infty$.
For example, if $b$ is non-extreme, then $\cH(b)$ contains the polynomials and they are dense
in $\cH(b)$. On the other hand, if $b$ is extreme, then the polynomials in $\cH(b)$ only span a
finite-dimensional subspace. In this article, we shall be concerned exclusively with the case where $b$
is non-extreme.

\subsection{Pythagorean pairs}
We say that $(b,a)$ is a \emph{Pythagorean pair} if
\begin{itemize}
\item $b,a\in H^\infty$;
\item $|b|^2+|a|^2=1$ a.e.\ on $\TT$;
\item $a$ is outer and $a(0)>0$.
\end{itemize}
The function $a$ is completely determined by $b$, and is 
called its \emph{Pythagorean mate}.
A function $b$ 
occurs as the first element of a Pythagorean pair 
iff it is a non-extreme point of the unit ball of $H^\infty$.
Thus the map $(b,a)\mapsto b$ is a bijection between the set of
Pythagorean pairs and the set of non-extreme points of the unit
ball of $H^\infty$.

Given a Pythagorean pair $(b,a)$, we write
$T_{\overline{b}}$ and $T_{\overline{a}}$ for the bounded Toeplitz operators on $H^2$
with symbols $\overline{b}$ and $\overline{a}$ respectively.
The following characterization of $\cH(b)$ will be fundamental in what follows.

\begin{theorem}[\protect{\cite[\S IV-1]{Sa94}}]\label{T:TbTa}
Let $(b,a)$ be a Pythagorean pair, and let $f\in H^2$.
Then $f\in\cH(b)$ iff there exists $g\in H^2$ such that $T_{\overline{b}}f=T_{\overline{a}}g$.
In this case, $g$ is uniquely determined, and we have
\begin{equation}\label{E:TbTa}
\|f\|_{\cH(b)}^2=\|f\|_{H^2}^2+\|g\|_{H^2}^2.
\end{equation}
\end{theorem}

\subsection{The Smirnov class}
The Smirnov class $N^+$ consists of those holomorphic functions
$\phi$ on $\DD$ that can be represented as
$\phi=\psi/\chi$, where $\psi,\chi\in H^\infty$ and $\chi$ is outer.
It is an algebra containing all the Hardy spaces $H^p~(0<p\le\infty)$.
For more about $N^+$ and $H^p$ we refer to
Duren's book \cite{Du70}.

Evidently, if $(b,a)$ is a Pythagorean pair, then $b/a\in N^+$.
Conversely, every $\phi\in N^+$ has a unique representation
$\phi=b/a$, where $(b,a)$ is a Pythag\-orean pair
(see e.g.\ \cite[Proposition~3.1]{Sa08}).
Thus the map $(b,a)\mapsto b/a$ is a bijection
between the set of Pythagorean pairs and the Smirnov class.

The correspondence between $(b,a)$ and $\phi$ comes into play
when considering formulas for $\|f\|_{\cH(b)}$ expressed in terms of the
Taylor coefficients of~$f$.
The following  result
is implicit in the work of Sarason \cite{Sa08}
and was formally established  in \cite[Theorem~4.1]{CGR10}.

\begin{theorem}\label{T:CGR}
Let $(b,a)$ be a Pythagorean pair and let $\phi:=b/a$.
If  $f$ is holomorphic on an open neighbourhood of $\overline{\DD}$, then $f\in\cH(b)$ and 
\[
\|f\|_{\cH(b)}^2=\sum_{m\ge0}|\hat{f}(m)|^2
+\sum_{m\ge0}\Bigl|\sum_{n\ge0}\overline{\hat{\phi}(n)}\hat{f}(m+n)\Bigr|^2.
\]
\end{theorem}

As explained in the introduction, the goal of this article is to extend 
Theorem~\ref{T:CGR} to a wider range of functions $f$.

%%%%%%%%%%%%%%%%%%%%%%%%%%%%%%%

\section{The case when $\phi\in H^2$}\label{S:H2}

The de Branges--Rovnyak spaces $\cH(b)$ for which $\phi\in H^2$
form a well-known subclass. Indeed, from \cite[Theorem~1]{Sa86}, we have
\[
\phi\in H^2
\iff \frac{1}{1-|b|^2}\in L^1(\TT)
\iff\sup_{n\ge0}\|z^n\|_{\cH(b)}<\infty
\iff H^\infty\subset \cH(b).
\]
A recent result in \cite[Theorem~1.3]{MS24} shows that we can replace $H^\infty$
in this last statement by any one of the spaces BMOA, VMOA (holomorphic functions of bounded/vanishing mean oscillation),  or
$A(\DD)$ (the disk algebra).
The condition $\phi\in H^2$  also characterizes the de Branges--Rovnyak spaces that admit
a reverse Carleson measure \cite[Corollary~4.4]{BFGHR15}.

Our goal in this section is to establish the following theorem.
In this result, and those that follow, we adopt
the convention that, if $f$ is holomorphic on $\DD$, then $\hat{f}(j)=0$ for all $j<0$.

\begin{theorem}\label{T:H2}
Let $(b,a)$ be a Pythagorean pair such that $\phi:=b/a\in H^2$,
and let $f\in H^2$.
\begin{enumerate}[\normalfont(i)]
\item For each $m\in\ZZ$ we have 
\[
\sum_{n\ge0}|\overline{\hat{\phi}(n)}\hat{f}(m+n)|<\infty.
\]
\item We have $f\in\cH(b)$ if and only if
\begin{equation}\label{E:H2cond}
\sum_{m\ge0}\Bigl|\sum_{n\ge0} \overline{\hat{\phi}(n)}\hat{f}(m+n)\Bigr|^2<\infty,
\end{equation}
in which case
\begin{equation}\label{E:H2formula}
\|f\|_{\cH(b)}^2=\sum_{m\ge0}|\hat{f}(m)|^2
+\sum_{m\ge0}\Bigl|\sum_{n\ge0}\overline{\hat{\phi}(n)}\hat{f}(m+n)\Bigr|^2.
\end{equation}
\end{enumerate}
\end{theorem}

\begin{proof}
(i)
For each $m\in\ZZ$, Cauchy--Schwarz gives
\begin{align*}
\sum_{n\ge0}|\overline{\hat{\phi}(n)}\hat{f}(m+n)|
&\le \Bigl(\sum_{n\ge0}|\hat{\phi}(n)|^2\Bigr)^{1/2}\Bigl(\sum_{n\ge0}|\hat{f}(n)|^2\Bigr)^{1/2}\\
&=\|\phi\|_{H^2}\|f\|_{H^2}<\infty.
\end{align*}
This proves (i). For future use, we note also that $\overline{\phi}f\in L^1(\TT)$ and that
\[
\hat{\overline{\phi}f}(m)=\sum_{n\ge0}\overline{\hat{\phi}(n)}\hat{f}(m+n)\quad(m\in\ZZ).
\]

(ii) 
Suppose first that  \eqref{E:H2cond} holds.
Set 
\[
g(z):=\sum_{m\ge0}\Bigl(\sum_{n\ge0}\overline{\hat{\phi}(n)}\hat{f}(m+n) \Bigr)z^m \quad(z\in\DD).
\]
Then  $g\in H^2$, 
and so $(\overline{\phi}f-g)\in L^1(\TT)$.
Also $\hat{g}(m)=\hat{\overline{\phi}f}(m)$
for all $m\ge0$, so the non-negative Fourier coefficients
of $(\overline{\phi}f-g)$ all vanish.
Therefore there exists $h\in H^1_0$ such that
\[
\overline{\phi}f-g=\overline{h} \quad\text{a.e. on~}\TT.
\]
Multiplying up by $\overline{a}$, we get
\[
\overline{b}f-\overline{a}g=\overline{ah} \quad\text{a.e. on~}\TT.
\]
Then $ah\in H^1_0\cap L^2(\TT)=H^2_0$, so $P_+(\overline{ah})=0$, where $P_+:L^2\to H^2$ is the canonical projection.
We deduce that $T_{\overline{b}}(f)=T_{\overline{a}}(g)$.
By Theorem~\ref{T:TbTa}, it follows that $f\in \cH(b)$ and that
\[
\|f\|_{\cH(b)}^2=\|f\|_{H^2}^2+\|g\|_{H^2}^2
=\sum_{m\ge0}|\hat{f}(m)|^2+\sum_{m\ge0}\Bigl|\sum_{n\ge0}\overline{\hat{\phi}(n)}\hat{f}(m+n)\Bigr|^2,
\]
which establishes \eqref{E:H2formula}.

Conversely, suppose that $f\in\cH(b)$. 
By Theorem~\ref{T:TbTa}, there exists $g\in H^2$
such that $T_{\overline{b}}(f)=T_{\overline{a}}(g)$. This means that there exists $h\in H^2_0$ such that
\[
\overline{b}f-\overline{a}g=\overline{h}\quad\text{a.e. on~}\TT.
\]
Dividing through by $\overline{a}$, we get
\[
\overline{\phi}f-g=\overline{h/a}\quad\text{a.e. on~}\TT.
\]
Now $h\in H^2$ and $a$ is a bounded outer function, 
so $h/a\in N^+$. Also, since $\phi,f,g\in L^2(\TT)$,
we have 
$(\overline{\phi}f-g)\in L^1(\TT)$, and thus $h/a\in L^1(\TT)$. 
By the Smirnov maximum principle
(see e.g.\ \cite[Theorem~2.11]{Du70}),
it follows that $h/a\in H^1$.  Further,
since $h(0)=0$, we in fact have $h/a\in H^1_0$.
Therefore the non-negative Fourier coefficients of
$(\overline{\phi}f-g)$ all vanish, and so
\[
\hat{g}(m)=\hat{\overline{\phi}f}(m)=\sum_{n\ge0}\overline{\hat{\phi}(n)}\hat{f}(m+n) \quad(m\ge0).
\]
As $g\in H^2$, we deduce that  \eqref{E:H2cond} holds.
\end{proof}

%%%%%%%%%%%%%%%%%%%%%%%%%%%%%%%%

\section{The case when $\phi\in H^p$}\label{S:Hp}

Let $(b,a)$ be a Pythagorean pair and let $\phi:=b/a$.
If $0<p\le \infty$, then 
\[
\phi\in H^p\iff (1-|b|^2)^{-1/2}\in L^p(\TT).
\]
Our goal in this section is to broaden the scope
of Theorem~\ref{T:H2} to cover the case when $\phi\in H^p$
where $0<p\le 2$. 

\begin{theorem}\label{T:Hp}
Let $p\in(0,2]$, and
let $(b,a)$ be a Pythagorean pair such that $\phi:=b/a\in H^p$.
Let $f$ be a holomorphic function on $\DD$ whose Taylor coefficients satisfy
\begin{equation}\label{E:Hphyp}
\sum_{m\ge0}m^{2/p-1}|\hat{f}(m)|^2<\infty.
\end{equation}
\begin{enumerate}[\normalfont(i)]
\item For all $m\in\ZZ$, we have
\[
\sum_{n\ge0}\Bigl|\overline{\hat{\phi}(n)}\hat{f}(m+n)\Bigr|<\infty.
\]
\item If 
\begin{equation}\label{E:Hpcond}
\sum_{m\ge0}\Bigl|\sum_{n\ge0}\overline{\hat{\phi}(n)}\hat{f}(m+n)\Bigr|^2<\infty,
\end{equation}
then $f\in\cH(b)$ and
\begin{equation}\label{E:Hpformula}
\|f\|_{\cH(b)}^2=\sum_{m\ge0}|\hat{f}(m)|^2
+\sum_{m\ge0}\Bigl|\sum_{n\ge0}\overline{\hat{\phi}(n)}\hat{f}(m+n)\Bigr|^2.
\end{equation}
\end{enumerate}
\end{theorem}

\begin{remark}
Theorem~\ref{T:Hp} is less satisfactory than
Theorem~\ref{T:H2} since it requires the extra hypothesis
\eqref{E:Hphyp} and, even with this hypothesis, it does not assert that the condition \eqref{E:Hpcond} holds for all $f\in\cH(b)$.
On the other hand, as we shall see later,
we cannot simply replace the assumption
$\phi\in H^2$ in Theorem~\ref{T:H2} by $\phi\in H^p$ or even by
$\phi\in\cap_{p<2}H^p$ (see Corollary~\ref{C:p<2} below).
\end{remark}

The proof of Theorem~\ref{T:Hp} makes use of
distributions on the unit circle~$\TT$,
so, before embarking upon the proof, we review
some basic facts about these distributions that will be needed
in what follows.

A \emph{distribution} on $\TT$ is a linear functional $u:C^\infty(\TT)\to\CC$
such that, for some constants $C,k$,
\begin{equation}\label{E:distcond}
|\langle u,\psi\rangle|\le C\sum_{j=0}^k\|\psi^{(j)}\|_{L^\infty(\TT)}
\quad(\psi\in C^\infty(\TT)).
\end{equation}
A sequence $(u_j)$ of distributions \emph{converges in distribution} to $u$ if
$\langle u_j,\psi\rangle\to\langle u,\psi\rangle$
for each $\psi\in C^\infty(\TT)$.
The \emph{Fourier coefficients} of a distribution $u$ on $\TT$ are defined by
\[
\hat{u}(n):=\langle u,e^{-in\theta}\rangle \quad(n\in\ZZ).
\]
If $u$ satisfies \eqref{E:distcond}, then $\hat{u}(n)=O(|n|^k)$
as $|n|\to\infty$. A distribution $u$ can be recovered from its
Fourier coefficients via the formula
\[
\langle u,\psi\rangle=\sum_{n\in\ZZ}\hat{u}(n)\hat{\psi}(-n)
\quad(\psi\in C^\infty(\TT)).
\]

Let $\phi\in H^p$, where $0<p\le\infty$. Then $\phi$ can be identified
with a distribution via the formula
\[
\langle \phi,\psi\rangle :=\lim_{r\to1^-}\frac{1}{2\pi}\int_\TT \phi(re^{i\theta})\psi(e^{i\theta})\,d\theta 
\quad(\psi\in C^\infty(\TT)).
\]
The non-negative Fourier coefficients of $\phi$ then coincide with its Taylor coefficients (so there is no ambiguity in the notation $\hat{\phi}(n)$). The negative Fourier coefficients are all zero.
 
For each $s\in\RR$, the \emph{Sobolev space} $W^{s,2}(\TT)$
consists of those distributions $u$ on $\TT$ such that
\[
\|u\|_{W^{s,2}(\TT)}^2:=\sum_{n\in\ZZ}(1+n^2)^{s}|\hat{u}(n)|^2<\infty.
\]
Hardy spaces and Sobolev spaces are linked by
the following lemma.

\begin{lemma}\label{L:HardySobolev}
Let $p\in(0,2]$ and let $s=1/p-1/2$. 
Then $H^p\subset W^{-s,2}(\TT)$,
and there exists a constant $C_p$ such that
\begin{equation}\label{E:HardySobolev}
\|\phi\|_{W^{-s,2}(\TT)}\le C_p\|\phi\|_{H^p} \quad(\phi\in H^p).
\end{equation}
\end{lemma}

\begin{proof}
For $p\in[1,2]$, this is a consequence of \cite[Theorem~1]{He68/69}.
For $p\in(0,1)$ it follows from \cite[Theorem~6.6]{Du70}.
\end{proof}

If $u,v\in L^2(\TT)$, then $uv\in L^1(\TT)$ and its Fourier coefficients are given by
\[
\hat{uv}(m)=\sum_{n\in \ZZ}\hat{u}(n)\hat{v}(m-n) \quad(m\in\ZZ).
\]
If $u\in W^{s,2}(\TT)$ and $v\in W^{-s,2}(\TT)$
for some $s\in[0,\infty)$, then we can reverse-engineer this
last formula to define the product $uv$ as a distribution on $\TT$.
The details are elaborated in the following lemma.

\begin{lemma}\label{L:prod}
Let $s\in[0,\infty)$, let $u\in W^{s,2}(\TT)$ and $v\in W^{-s,2}(\TT)$.
\begin{enumerate}[\normalfont(i)]
\item For each $m\in\ZZ$, we have
\begin{equation}\label{E:convest}
\sum_{n\in\ZZ}|\hat{u}(n)\hat{v}(m-n)|
\le 2^{s/2}(1+m^2)^{s/2}\|u\|_{W^{s,2}(\TT)}\|v\|_{W^{-s,2}(\TT)}.
\end{equation}
\item There exists a unique distribution $uv$ on $\TT$ such that
\begin{equation}\label{E:prodcoeffs}
\hat{uv}(m)=\sum_{n\in \ZZ}\hat{u}(n)\hat{v}(m-n) \quad(m\in\ZZ).
\end{equation}
\item
There exists a constant $C_s$ such that, for all $\psi\in C^\infty(\TT)$,
\begin{equation}\label{E:prodest}
|\langle uv,\psi\rangle|\le C_s\|u\|_{W^{s,2}(\TT)}\|v\|_{W^{-s,2}(\TT)}\sum_{0\le j\le s+2}\|\psi^{(j)}\|_{L^\infty(\TT)}.
\end{equation}
\end{enumerate}
\end{lemma}

\begin{proof}
By Cauchy--Schwarz, for each $m\in\ZZ$ we have
\begin{align*}
&\sum_{n\in\ZZ}| \hat{u}(n)\hat{v}(m-n)|\\
&\le \Bigl(\sum_{n\in\ZZ}(1+n^2)^s|\hat{u}(n)|^2\Bigr)^{1/2}
\Bigl(\sum_{n\in\ZZ}(1+n^2)^{-s}|\hat{v}(m-n)|^2\Bigr)^{1/2}\\
&\le\sup_{n\in\ZZ}\frac{(1+n^2)^{-s/2}}{(1+(m-n)^2)^{-s/2}}\|u\|_{W^{s,2}(\TT)}\|v\|_{W^{-s,2}(\TT)}\\
&\le 2^{s/2}(1+m^2)^{s/2}\|u\|_{W^{s,2}(\TT)}\|v\|_{W^{-s,2}(\TT)}.
\end{align*}
This establishes (i).
We shall prove (ii) and (iii) together.
Define a linear functional $uv:C^\infty(\TT)\to\CC$ by
\[
\langle uv,\psi\rangle:=\sum_{m\in\ZZ}\Bigl(\sum_{n\in\ZZ}\hat{u}(n)\hat{v}(m-n)\Bigr)\hat{\psi}(-m) \quad(\psi\in C^\infty(\TT)).
\]
Then, for each $\psi\in C^\infty(\TT)$, we have
\begin{align*}
|\langle uv,\psi\rangle|
&\le \sum_{m\in\ZZ}\Bigl|\sum_{n\in\ZZ}\hat{u}(n)\hat{v}(m-n)\Bigr||\hat{\psi}(-m)|\\
&\le\sum_{m\in\ZZ}2^{s/2}(1+m^2)^{s/2}\|u\|_{W^{s,2}(\TT)}\|v\|_{W^{-s,2}(\TT)}|\hat{\psi}(-m)|
\end{align*}
and, if $k$ is the smallest integer such that $k>s+1/2$, then
\begin{align*}
&\sum_{m\in\ZZ}(1+m^2)^{s/2}|\hat{\psi}(-m)|\\
&\le|\hat{\psi}(0)|+\Bigl(\sum_{m\in\ZZ\setminus\{0\}}\frac{(1+m^2)^s}{m^{2k}}\Bigr)^{1/2}\Bigl(\sum_{m\in\ZZ\setminus\{0\}}|m^{k}\hat{\psi}(-m)|^2\Bigr)^{1/2}\\
&\le \|\psi\|_{L^2(\TT)}+\Bigl(\sum_{m\in\ZZ\setminus\{0\}}\frac{(1+m^2)^s}{m^{2k}}\Bigr)^{1/2}\|\psi^{(k)}\|_{L^2(\TT)}.
\end{align*}
Combining these observations leads quickly to the estimate \eqref{E:prodest}.
It also shows that $uv$ is a distribution on $\TT$. Clearly
the Fourier coefficients of $uv$ are given by~\eqref{E:prodcoeffs}, and they determine $uv$ uniquely.
\end{proof}

\begin{proof}[Proof of Theorem~\ref{T:Hp}]
(i) Let $s:=1/p-1/2$, so $s\ge0$.
By Lemma~\ref{L:HardySobolev}, we have $\phi\in W^{-s,2}(\TT)$, 
and by the hypothesis \eqref{E:Hphyp},
we have $f\in W^{s,2}(\TT)$. Hence, by Lemma~\ref{L:prod}\,(i), 
for each $m\in\ZZ$ we have
\[
\sum_{n\ge0}|\overline{\hat{\phi}(n)}\hat{f}(m+n)|
\le 2^{s/2}(1+m^2)^{s/2}\|\phi\|_{W^{-s,2}(\TT)}\|f\|_{W^{s,2}(\TT)}<\infty.
\]
This proves (i). We also  note for future use that $\overline{\phi}f$ is a distribution 
on $\TT$ satisfying
\[
\hat{\overline{\phi}f}(m)=\sum_{n\in\ZZ}\overline{\hat{\phi}(n)}\hat{f}(m+n) \quad (m\in\ZZ).
\]

(ii) The argument here is broadly similar to that in the proof of Theorem~\ref{T:H2}, but more care is required because 
some of the terms being manipulated 
are distributions rather than functions.

Suppose that the condition \eqref{E:Hpcond} holds.
Set 
\[
g(z):=\sum_{m\ge0}\Bigl(\sum_{n\ge0}\overline{\hat{\phi}(n)}\hat{f}(m+n) \Bigr)z^m \quad(z\in\DD).
\]
Then  $g\in H^2$ and $\hat{g}(m)=\hat{\overline{\phi}f}(m)$
for all $m\ge0$, so $(\overline{\phi}f-g)$ is a distribution on $\TT$
whose non-negative Fourier coefficients all vanish.
Let $r\in(0,1)$ and let $a_r$ be the $r$-dilation of $a$.
Since $a_r$ is holomorphic on a neighbourhood of~$\overline{\DD}$, we can multiply by $\overline{a}_r$
to obtain a distribution $\overline{a}_r(\overline{\phi}f-g)$ that still has
the property that its non-negative Fourier coefficients vanish.
We claim that $\overline{a}_r(\overline{\phi}f-g)$ converges in
distribution to $(\overline{b}f-\overline{a}g)$ as $r\to1$.
If so, then the non-negative Fourier coefficients of
$(\overline{b}f-\overline{a}g)$  all vanish. Since
$(\overline{b}f-\overline{a}g)\in L^2(\TT)$, it follows that
$P_+(\overline{b}f-\overline{a}g)=0$, where $P_+:L^2(\TT)\to H^2$
is the canonical projection. We deduce that $T_{\overline{b}}(f)=T_{\overline{a}}(g)$.
By Theorem~\ref{T:TbTa}, it follows that $f\in \cH(b)$ and that
\[
\|f\|_{\cH(b)}^2=\|f\|_{H^2}^2+\|g\|_{H^2}^2
=\sum_{m\ge0}|\hat{f}(m)|^2+\sum_{m\ge0}\Bigl|\sum_{n\ge0}\overline{\hat{\phi}(n)}\hat{f}(m+n)\Bigr|^2,
\]
which establishes \eqref{E:Hpformula}.

It remains to justify the claim that 
$\overline{a}_r(\overline{\phi}f-g)$ converges in
distribution to $(\overline{b}f-\overline{a}g)$ as $r\to1$.
Since $g\in H^2$ and $a\in H^\infty$, the
dominated convergence theorem implies 
 that $\|\overline{a}_rg-\overline{a}g\|_{L^2(\TT)}\to0$, 
 so in particular $\overline{a}_rg$ converges 
 in distribution to $\overline{a}g$. 
 The proof that $\overline{a}_r(\overline{\phi}f)$ converges
 in distribution to $\overline{b}f$ is more delicate.
 The first step is to show that 
 $\overline{a}_r(\overline{\phi}f)=(\overline{a_r\phi})f$.
 For this, it suffices to check that the two sides have the
 same Fourier coefficients, and a calculation shows that
 this boils down to showing that a certain double series commutes.
 Since $a_r$ is holomorphic on a neighbourhood of $\overline{\DD}$,
 its Fourier coefficients decay exponentially, which implies that
 the double series in question converges absolutely, and hence it does commute.
 So indeed 
 $\overline{a}_r(\overline{\phi}f)=(\overline{a_r\phi})f$
 for each $r<1$. Next,  by the dominated convergence theorem,
 we have $\|(a_r-a)\phi\|_{L^p(\TT)}\to0$ as $r\to1$.
 As  $a,a_r\in H^\infty$, $\phi\in H^p$ and $a\phi=b$, 
 this implies that $\|a_r\phi-b\|_{H^p}\to0$.
 From \eqref{E:HardySobolev} it follows that
 $\|\overline{a_r\phi}-\overline{b}\|_{W^{-s,2}(\TT)}\to0$.
 By \eqref{E:prodest}, this in turn implies that 
 $\langle (\overline{a_r\phi}-\overline{b})f,\psi\rangle\to0$
 for each $\psi\in C^\infty(\TT)$, in other words,
 that $(\overline{a_r\phi})f$ 
 converges in distribution to~$\overline{b}f$.
 This completes the justification of the claim 
 and with it, the proof of the
 theorem.
\end{proof}

%%%%%%%%%%%%%%%%%%%%%%%%%%%%%%%

\section{The case when $\phi$ is rational}\label{S:rational}

The de Branges--Rovnyak spaces $\cH(b)$ with 
rational $b$ have been studied by a number of authors.
References include \cite{CR13, FHR16, LMNS24,LN17,LGR21}.
In particular, there are interesting connections with Dirichlet-type spaces,
which were first noticed by Sarason \cite{Sa97}.

Let $(b,a)$ be a Pythagorean pair and let $\phi=b/a$. Then
\[
b \text{~is rational~}\iff (b,a) \text{~are both rational~}\iff \phi \text{~is rational}.
\]
In this case, the zeros of $a$ on $\TT$ are exactly the poles of
$\phi$ on $\TT$, with the same multiplicities. 
Listing them (with multiplicity) as $\lambda_1,\dots,\lambda_k$,
we have
\begin{equation}\label{E:rationalH(b)}
f\in \cH(b)\iff f(z)=g(z)\prod_{j=1}^k(z-\lambda_j)+p(z),
\end{equation}
where $g\in H^2$ and $p$ is a polynomial of degree at most $k-1$
(see \cite[Lemma~4.3]{CR13}).

Our goal in this section is to establish the following theorem.

\begin{theorem}\label{T:rational}
Let $(b,a)$ be a rational Pythagorean pair, and let $\phi:=b/a$.
\begin{enumerate}[\normalfont(i)]
\item If  $f\in H^2$ and
\begin{equation}\label{E:rationalcond}
\sum_{m\ge0}\Bigl|\sum_{n\ge0}\overline{\hat{\phi}(n)}\hat{f}(m+n)\Bigr|^2<\infty,
\end{equation}
then $f\in\cH(b)$ and
\begin{equation}\label{E:rationalformula}
\|f\|_{\cH(b)}^2=\sum_{m\ge0}|\hat{f}(m)|^2
+\sum_{m\ge0}\Bigl|\sum_{n\ge0}\overline{\hat{\phi}(n)}\hat{f}(m+n)\Bigr|^2.
\end{equation}
\item If all the poles of $\phi$ on $\TT$ are simple, and if
$f\in\cH(b)$, then \eqref{E:rationalcond} holds.
\item If $\phi$ has a multiple pole on $\TT$,
then there exist $f\in\cH(b)$ and $m\ge0$ such that the series
\[
\sum_{n\ge0} \overline{\hat{\phi}(n)}\hat{f}(m+n)
\]
diverges, and consequently \eqref{E:rationalcond} and \eqref{E:rationalformula} fail to hold
for this $f$.
\end{enumerate}
\end{theorem}

\begin{remark}
It is implicit in \eqref{E:rationalcond}
that each individual series 
$\sum_{n\ge0}\overline{\hat{\phi}(n)}\hat{f}(m+n)$
converges (not necessarily absolutely).
\end{remark}

For  the proof of Theorem~\ref{T:rational},
we use the following lemma.

\begin{lemma}\label{L:bdedcoeffs}
Let $\rho(z)$ be a rational function with poles outside $\DD$.
\begin{enumerate}[\normalfont(i)]
\item If there exists $\alpha\in(0,1)$ such that $\hat{\rho}(j)=O(j^\alpha)$ as $j\to\infty$,
then all the poles of $\rho$ on $\TT$ are simple.
\item Conversely, if all the poles of $\rho$ on $\TT$ are simple, then
 $\sup_j|\hat{\rho}(j)|<\infty$.
 \end{enumerate}
\end{lemma}

\begin{proof}
(i) Suppose that $|\hat{\rho}(j)|\le C(1+j)^\alpha$ for all $j\ge0$.
Applying H\"older's inequality with $p=1/\alpha$ and $q=1/(1-\alpha)$, we see that, for all $z\in\DD$,
\begin{align*}
|\rho(z)|&
\le \sum_{j\ge0}C(1+j)^\alpha|z|^j
\le C\Bigl(\sum_{j\ge0}(1+j)|z|^j\Bigr)^\alpha\Bigl(\sum_{j\ge0}|z|^j\Bigr)^{1-\alpha}\\
&=C(1-|z|)^{-2\alpha}(1-|z|)^{-(1-\alpha)}=C(1-|z|)^{-1-\alpha}.
\end{align*}
Evidently, this estimate implies that all the poles of $\rho$ on $\TT$ are simple.

(ii) Suppose that all the poles of $\rho$ on $\TT$
are simple.
Expanding $\rho(z)$ in partial fractions, we can write it
as a polynomial plus a linear combination of fractions of the form 
$1/(z-\lambda)~(|\lambda|=1)$ and $1/(z-\mu)^k~(|\mu|>1, k\ge1)$,
each of which has a Taylor expansion with bounded coefficients.
Hence the Taylor coefficients of $\rho(z)$ are bounded.
\end{proof}

\begin{proof}[Proof of Theorem~\ref{T:rational}]
We begin with some preliminary remarks that apply to all
parts of the proof.
As $\phi$ is rational, we may write it as $\phi:=p/q$, 
where $p,q$ are polynomials with no common zeros. 
By a well-known result of Fej\'er,
there exists a polynomial $r$ with all its zeros
outside $\overline{\DD}$ such that
$|p|^2+|q|^2=|r|^2$ on $\TT$.
Then $(p/r,\,q/r)$ is a Pythagorean pair such that
$(p/r)/(q/r)=\phi$, so, by uniqueness of $(b,a)$, we must have
$p/r=b$ and $q/r=a$.
The poles of $\phi$ on $\TT$ are exactly the zeros
of $q$ on $\TT$, which in turn are exactly the zeros of $a$ on $\TT$.

(i) Suppose that $f\in H^2$ and that \eqref{E:rationalcond} holds. 
Set 
\[
g(z):
=\sum_{m\ge0}\Bigl(\sum_{n\ge0}\overline{\hat{\phi}(n)}\hat{f}(m+n) \Bigr)z^m
\quad(z\in\DD).
\]
Then $g\in H^2$. Also, for each $m\ge0$, we have
\begin{align*}
\sum_{j=0}^{\deg q}\overline{\hat{q}(j)}\hat{g}(j+m)
&=\sum_{j=0}^{\deg q}\overline{\hat{q}(j)}\Bigl(\sum_{n\ge j}\overline{\hat{\phi}(n-j)}\hat{f}(m+n)\Bigr)\\
&=\sum_{n\ge0}\sum_{j=0}^{\min\{n,\deg q\}}
\overline{\hat{q}(j)}\overline{\hat{\phi}(n-j)}\hat{f}(m+n)\\
&=\sum_{n=0}^{\deg p}\overline{\hat{p}(n)}\hat{f}(m+n).
\end{align*}
Hence $T_{\overline{q}}g=T_{\overline{p}}f$,
in other words,  $\overline{p}f-\overline{q}g=\overline{h}$ on $\TT$,
where $h\in H^2_0$.
Dividing through by $\overline{r}$, we obtain $\overline{b}f-\overline{a}g=\overline{h/r}$ on $\TT$, where $h/r\in H^2_0$.
This implies that $T_{\overline{b}}f=T_{\overline{a}}g$.
Using Theorem~\ref{T:TbTa}, we deduce that $f\in\cH(b)$ 
and that
\[
\|f\|_{\cH(b)}^2=\|f\|_{H^2}^2+\|g\|_{H^2}^2
=\sum_{m\ge0}|\hat{f}(m)|^2+\sum_{m\ge0}\Bigl|\sum_{n\ge0}\overline{\hat{\phi}(n)}\hat{f}(m+n)\Bigr|^2,
\]
which establishes \eqref{E:rationalformula}.

(ii) Suppose now that all the poles of $\phi$ on $\TT$
are simple, or equivalently, that all the zeros of $q$ on $\TT$ are simple.  
Let $f\in\cH(b)$.
By Theorem~\ref{T:TbTa}, there exists $g\in H^2$
such that $T_{\overline{b}}f=T_{\overline{a}}g$,
in other words, $\overline{b}f-\overline{a}g=\overline{h}$ on $\TT$,
where $h\in H_0^2$. Multiplying both sides by $\overline{r}$,
we obtain $\overline{p}f-\overline{q}g=\overline{rh}$
on $\TT$, where $rh\in H^2_0$. It follows that
$T_{\overline{p}}f=T_{\overline{q}}g$. 
In terms of coefficients, this says that
\begin{equation}\label{E:rational}
\sum_{n=0}^{\deg p} \overline{\hat{p}(n)}\hat{f}(m+n)
=\sum_{n=0}^{\deg q} \overline{\hat{q}(n)}\hat{g}(m+n)
\quad(m\ge0).
\end{equation}
By Lemma~\ref{L:bdedcoeffs}\,(ii), we can write
$1/q(z)=\sum_{j\ge0} c_jz^j$ for $z\in\DD$,
where $\sup_j|c_j|<\infty$.
Rewriting \eqref{E:rational} as 
\[
\sum_{n=j}^{j+\deg p} \overline{\hat{p}(n-j)}\hat{f}(m+n)
=\sum_{n=j}^{j+\deg q} \overline{\hat{q}(n-j)}\hat{g}(m+n)
\quad(j,m\ge0),
\]
multiplying both sides by $\overline{c}_j$ and summing 
from $j=0$ to $J$, we obtain
\[
\sum_{j=0}^J\overline{c}_j\sum_{n=j}^{j+\deg p} \overline{\hat{p}(n-j)}\hat{f}(m+n)
=\sum_{j=0}^J\overline{c}_j\sum_{n=j}^{j+\deg q} \overline{\hat{q}(n-j)}\hat{g}(m+n)
\quad(m,J\ge0).
\]
After rearrangement, this becomes
\begin{align*}
&\sum_{n=0}^{J+\deg p}\sum_{j=\max\{0,n-\deg p\}}^{\min\{J,n\}}\overline{c}_j \overline{\hat{p}(n-j)}\hat{f}(m+n)\\
&\quad=\sum_{n=0}^{J+\deg q}\sum_{j=\max\{0,n-\deg q\}}^{\min\{J,n\}}\overline{c}_j \overline{\hat{q}(n-j)}\hat{g}(m+n) 
\quad(m,J\ge0).
\end{align*}
Now, since $(1/q(z))p(z)=\phi(z)$, the coefficients $c_j$ satisfy
\[
\sum_{j=\max\{0,n-\deg p\}}^n c_j\hat{p}(n-j)=\hat{\phi}(n)
\quad(n\ge0).
\]
Therefore, for all $m,J\ge0$, we have
\begin{align*}
&\sum_{n=0}^{J+\deg p}\sum_{j=\max\{0,n-\deg p\}}^{\min\{J,n\}}\overline{c}_j \overline{\hat{p}(n-j)}\hat{f}(m+n)\\
&\quad=\sum_{n=0}^{J}\sum_{j=\max\{0,n-\deg p\}}^{n}\overline{c}_j \overline{\hat{p}(n-j)}\hat{f}(m+n)\\
&\qquad+\sum_{n=J+1}^{J+\deg p}\sum_{j=\max\{0,n-\deg p\}}^{J}\overline{c}_j \overline{\hat{p}(n-j)}\hat{f}(m+n)\\
&\quad=\sum_{n=0}^J\overline{\hat{\phi}(n)}\hat{f}(m+n)+B_{J,m},
\end{align*}
where
\[
|B_{J,m}|\le (\deg p)^2(\sup_j|c_j|)(\sup_k|\hat{p}(k)|)(\sup_{n\ge J}|\hat{f}(n)|)=o(1) \quad(J\to\infty).
\]
Likewise, since $(1/q(z))q(z)=1$, the coefficients $c_j$ satisfy
\[
\sum_{j=\max\{0,n-\deg q\}}^n c_j\hat{q}(n-j)=
\begin{cases}
1, &n=0,\\
0, &n>0.
\end{cases}
\]
Therefore, for each $m,J\ge0$,
\begin{align*}
&\sum_{n=0}^{J+\deg q}\sum_{j=\max\{0,n-\deg q\}}^{\min\{J,n\}}\overline{c}_j \overline{\hat{q}(n-j)}\hat{g}(m+n)\\
&\quad=\sum_{n=0}^{J}\sum_{j=\max\{0,n-\deg q\}}^{n}\overline{c}_j \overline{\hat{q}(n-j)}\hat{g}(m+n)\\
&\qquad+\sum_{n=J+1}^{J+\deg q}\sum_{j=\max\{0,n-\deg q\}}^{J}\overline{c}_j \overline{\hat{q}(n-j)}\hat{g}(m+n)\\
&\quad=\hat{g}(m)+A_{J,m},
\end{align*}
where
\[
|A_{J,m}|\le (\deg q)^2(\sup_j|c_j|)(\sup_k|\hat{q}(k)|)(\sup_{n\ge J}|\hat{g}(n)|)=o(1) \quad(J\to\infty).
\]
Combining these observations, we deduce that
\[
\sum_{n\ge0} \overline{\hat{\phi}(n)}\hat{f}(m+n)=\hat{g}(m)
\quad(m\ge0).
\]
Thus
\[
\sum_{m\ge0}\Bigl|\sum_{n\ge0} \overline{\hat{\phi}(n)}\hat{f}(m+n)\Bigr|^2
=\sum_{m\ge0}|\hat{g}(m)|^2=\|g\|_{H^2}^2<\infty,
\]
and so \eqref{E:rationalcond} holds.

(iii) Let $\lambda_1,\dots,\lambda_k$ 
be the poles of $\phi$ on $\TT$,
listed according to multiplicity.
Fix an integer $K>k$ and $\alpha\in(\frac{1}{2},1)$, and define
\[
g(z):=\sum_{n\ge1}\frac{z^{Kn}}{n^{\alpha}}
\quad\text{and}\quad
f(z):=g(z)\prod_{j=1}^k(1-\overline{\lambda}_jz).
\]
Then $g\in H^2$, so
we can apply the criterion \eqref{E:rationalH(b)}
to deduce that $f\in\cH(b)$.

Suppose that the series
\begin{equation}\label{E:series}
\sum_{n\ge 0}\overline{\hat{\phi}(n)}\hat{f}(m+n)
\end{equation}
converges for each $m\ge0$. Then certainly
$\lim_{n\to\infty}\overline{\hat{\phi}(n-m)}\hat{f}(n)=0$
for each $m\ge0$.
Now, because of the gaps in the series defining $g$, we have
\[
\hat{f}(Kn)=\hat{g}(Kn)=1/n^{\alpha} \quad(n\ge 1).
\]
It follows that
\[
\lim_{n\to\infty}\overline{\hat{\phi}(Kn-m)}/n^{\alpha}=0 \quad(m\ge0).
\]
This implies that $\hat{\phi}(n)=O(n^{\alpha})$ as $n\to \infty$.
By Lemma~\ref{L:bdedcoeffs}\,(i), 
it follows that all the poles of $\phi$ on~$\TT$
are simple. 
Thus, if $\phi$ has a multiple pole on $\TT$,
then the series \eqref{E:series} diverges 
for at least one value of $m\ge0$.
\end{proof}

%%%%%%%%%%%%%%%%%%%%%%%%%%%%%%%

\section{Counterexamples}\label{S:counterexamples}

Our goal in this section is to establish the following theorem,
which is a rich source of counterexamples.

\begin{theorem}\label{T:negative}
Let  $\phi_0\in N^+$ be a function that is real on the real axis
and satisfies $\lim_{r\to 1^-}\phi_0(r)=\infty$.
Let $B$ be an infinite Blaschke product whose zeros $(t_n)_{n\ge1}$ lie in $(0,1)$
and satisfy
\begin{equation}\label{E:zerocond}
\inf_{n\ge1}\frac{1-t_{n+1}}{1-t_n}>0
\quad\text{and}\quad
\sup_{n\ge1}\frac{1-t_{n+1}}{1-t_n}<1/2.
\end{equation}
Set $\phi(z):=B(z)^2\phi_0(z)/(1-z)^{1/2}$, and let $(b,a)$ be the Pythagorean
pair with $\phi=b/a$. Then there exists
$f\in\cH(b)$ such that 
$\sum_{n\ge0}\overline{\hat{\phi}(n)}\hat{f}(n)$
diverges. 
\end{theorem}

Before proving Theorem~\ref{T:negative},
we present two applications. 
The first of these is an example
showing that Theorem~\ref{T:H2} is sharp, in the sense
that one cannot replace
the hypothesis $\phi\in H^2$ by $\phi\in \cap_{p<2}H^p$.

\begin{corollary}\label{C:p<2}
There exist a Pythagorean pair $(b,a)$ and $f\in\cH(b)$ such that 
$\phi:=b/a\in\cap_{p<2}H^p$,
 yet $\sum_{n\ge0}\overline{\hat{\phi}(n)}\hat{f}(n)$ diverges.
\end{corollary}

\begin{proof}
Let $(b,a)$ be the Pythagorean pair corresponding to
the function
$\phi(z):=-B(z)^2\log(1-z)/(1-z)^{1/2}$,
where $B$ is a Blaschke product whose zeros
$(t_n)_{n\ge1}$ lie in $(0,1)$ and
satisfy~\eqref{E:zerocond}.
By Theorem~\ref{T:negative}, applied with $\phi_0(z):=-\log(1-z)$,
there is an $f\in\cH(b)$ such that the series $\sum_{n\ge0}\overline{\hat{\phi}(n)}\hat{f}(n)$ diverges.
Also, we clearly have $\phi\in\cap_{p<2}H^p$. 
\end{proof}

The second application of Theorem~\ref{T:negative}
is an example showing that part~(ii) of Theorem~\ref{T:H2} may hold for certain
$\phi\notin H^2$ (even for $\phi\notin H^1$),
but that it may then be unstable
under multiplication by Blaschke products.

\begin{corollary}\label{C:H1}
There exist a Pythagorean pair $(b,a)$ 
and a Blaschke product~$B$ with the following properties:
\begin{itemize}
\item $\phi:=b/a\in N^+\setminus H^1$;
\item the formula \eqref{E:H2formula}  holds for all $f\in\cH(b)$;
\item the analogue of \eqref{E:H2formula} for $\cH(B^2b)$ fails  for at least one $f\in\cH(B^2b)$.
\end{itemize}
\end{corollary}

\begin{proof}
Let $\phi(z):=1/(1-z)$ and let $(b,a)$ be the Pythagorean pair
such that $\phi=b/a$. 
Clearly $\phi\in N^+\setminus H^1$.
By Theorem~\ref{T:rational},
the formula \eqref{E:H2formula}
holds for~$\cH(b)$.

Let $B$ be an infinite Blaschke product whose zeros $(t_n)_{n\ge1}$
lie in $(0,1)$ and satisfy \eqref{E:zerocond}.
Then $(B^2b,a)$ is a Pythagorean pair and 
$B(z)^2b(z)/a(z)=B(z)^2/(1-z)$,
so, by Theorem~\ref{T:negative}
applied with $\phi_0(z)=(1-z)^{-1/2}$,
the analogue of \eqref{E:H2formula} for $\cH(B^2b)$ 
fails  for at least one $f\in\cH(B^2b)$.
\end{proof}

We now turn to
the proof of Theorem~\ref{T:negative}.
This is strongly influenced by the arguments in \cite{EFKMR16}. We need three lemmas,
the first of which is a special case of \cite[Lemma~3.3]{EFKMR16}.

\begin{lemma}\label{L:Blaschke}
Let $B$ be a Blaschke product whose zeros $(t_n)_{n\ge1}$
lie in $(0,1)$ and satisfy \eqref{E:zerocond}.
Then 
\[
\inf_{n\ge1}|B(t_n^2)|>0.
\]
\end{lemma}
 
The second lemma is a standard property of coanalytic
Toeplitz operators.
We write $k_w(z):=1/(1-\overline{w}z)$,
the reproducing kernel of $H^2$ at $w\in\DD$.

\begin{lemma}\label{L:RK}
Let $\psi\in H^\infty$ and let $w\in\DD$.
Then
\[
T_{\overline{\psi}}k_w=\overline{\psi(w)}k_w.
\]
\end{lemma}

\begin{proof} 
For all $f\in H^2$, we have
\[
\langle f,T_{\overline{\psi}}k_w\rangle
=\langle \psi f,k_w\rangle
=\psi(w)f(w)
=\psi(w)\langle f,k_w\rangle
=\langle f,\overline{\psi(w)}k_w\rangle.\qedhere
\]
\end{proof}

The third lemma is a  calculation.

\begin{lemma}\label{L:calc}
Let $(b,a)$ be a Pythagorean pair, let $f\in H^2$, and let $r\in(0,1)$.
Then $T_{\overline{b}}(f_r)=T_{\overline{a}}(g)$, where $g\in H^2$
and
\[
g(0)=\sum_{n\ge0}\overline{\hat{\phi}(n)}\hat{f}(n)r^n.
\]
\end{lemma}

\begin{proof}
We shall  use the following simple observation
about $r$-dilations: if $u\in H^\infty$ and $v\in H^2$, and $P_+:L^2(\TT)\to H^2$
is  the canonical projection, then
\begin{equation}\label{E:obs}
P_+(\overline{u}v_r)=(P_+(\overline{u_r}v))_r \quad(0<r<1).
\end{equation}

Clearly we have $\overline{\phi_r}f\in L^2(\TT)$.
Define $g:=(P_+(\overline{\phi_r}f))_r$. Then $g\in H^2$.
Also, using the observation \eqref{E:obs} above,
we have
\begin{align*}
T_{\overline{a}}g
&=P_+\bigl(\overline{a}(P_+(\overline{\phi_r}f))_r\bigr)
=\bigl(P_+\bigl(\overline{a_r}(P_+(\overline{\phi_r}f))\bigr)\bigr)_r
=\bigl(P_+\bigl(\overline{a_r}\overline{\phi_r}f)\bigr)_r\\
&=\bigl(P_+\bigl(\overline{b_r}f)\bigr)_r
=P_+(\overline{b}f_r)=T_{\overline{b}}(f_r).
\end{align*}
Finally, we have
\[
g(0)=(P_+(\overline{\phi_r}f))_r(0)=\hat{\overline{\phi_r}f}(0)=
\sum_{n\ge0}\overline{\hat{\phi}(n)r^n}\hat{f}(n).\qedhere
\]
\end{proof}

\begin{proof}[Proof of Theorem~\ref{T:negative}]
Since $\phi_0(r)\to\infty$ as $r\to1^-$,
there exists $n_0\ge1$ such that
$\phi_0(r)\ge0$ for all $r\ge t_{n_0}^2$.
Also $\phi_0(t_n^2)\to\infty$ as $n\to\infty$,
so there exists a non-negative sequence 
$(c_n)_{n\ge n_0}$ such that
\begin{equation}\label{E:infcond}
\sum_{n\ge n_0} c_n<\infty
\quad\text{and}\quad
\sup_{n\ge n_0}c_n\phi_0(t_n^2)=\infty.
\end{equation}
Set
\[
f:=\sum_{n\ge n_0}c_n(1-t_n^2)^{1/2}k_{t_n}.
\]
We claim that $f$ is well-defined  and that it belongs to $\cH(b)$,
but that the series $\sum_{m\ge0}\overline{\hat{\phi}(m)}\hat{f}(m)$ diverges.

First we show that $f\in\cH(b)$.
By Lemma~\ref{L:RK},
 we have $T_{\overline{b}}k_{w}=\overline{b(w)}k_w$
 for all $w\in \DD$. Also,
for all $n$, we have $B(t_n)=0$, and so
\[
b(t_n)=a(t_n)\phi(t_n)=a(t_n)B(t_n)^2\phi_0(t_n)/(1-t_n)^{1/2}=0.
\]
Therefore $T_{\overline{b}}k_{t_n}=0$.
Applying \eqref{E:TbTa} with $f=k_{t_n}$ and $g=0$, we deduce that
\[
\|k_{t_n}\|_{\cH(b)}^2=\|k_{t_n}\|_{H^2}^2
=\frac{1}{1-t_n^2}.
\]
Thus
\[
\sum_{n\ge n_0}c_n(1-t_n^2)^{1/2}\|k_{t_n}\|_{\cH(b)}
=\sum_{n\ge n_0}c_n<\infty.
\]
Therefore the series defining $f$ converges absolutely in $\cH(b)$.

Now we show that $\sum_{m\ge0}\overline{\hat{\phi}(m)}\hat{f}(m)$ diverges.
Let $r\in(0,1)$. Then 
\[
f_r=\sum_{n\ge n_0}c_n(1-t_n^2)^{1/2}k_{rt_n},
\]
where the series converges absolutely in the norm of $H^2$.
Using Lemma~\ref{L:RK} once again, we have
\begin{align*}
T_{\overline{b}}f_r
&=\sum_{n\ge n_0}c_n(1-t_n^2)^{1/2}T_{\overline{b}}k_{rt_n}\\
&=\sum_{n\ge n_0}c_n(1-t_n^2)^{1/2}\overline{b(rt_n)}k_{rt_n}\\
&=\sum_{n\ge n_0}c_n(1-t_n^2)^{1/2}\overline{\phi(rt_n)a(rt_n)}k_{rt_n}\\
&=\sum_{n\ge n_0}c_n(1-t_n^2)^{1/2}\overline{\phi(rt_n)}T_{\overline{a}}k_{rt_n}\\
&=T_{\overline{a}}g,
\end{align*}
where
\[
g=\sum_{n\ge n_0}c_n(1-t_n^2)^{1/2}\overline{\phi(rt_n)}k_{rt_n},
\]
the series converging in $H^2$. 
By Theorem~\ref{T:TbTa}, $g$ is the unique function in~$H^2$
such that $T_{\overline{b}}(f_r)=T_{\overline{a}}g$,
and by Lemma~\ref{L:calc}
\[
g(0)=\sum_{m\ge0}\overline{\hat{\phi}(m)}\hat{f}(m)r^m.
\]
We deduce that
\begin{align*}
\sum_{m\ge0}\overline{\hat{\phi}(m)}\hat{f}(m)r^m
&=\sum_{n\ge n_0}c_n(1-t_n^2)^{1/2}\overline{\phi(rt_n)}\\
&=\sum_{n\ge n_0}c_n(1-t_n^2)^{1/2}B(rt_n)^2\phi_0(rt_n)(1-rt_n)^{-1/2}.
\end{align*}
If, further, $r\ge t_{n_0}$, then all the terms in this last series are
positive, and so, for each $n\ge n_0$,
\[
\sum_{m\ge0}\overline{\hat{\phi}(m)}\hat{f}(m)r^m
\ge c_n(1-t_n^2)^{1/2}B(rt_n)^2\phi_0(rt_n)(1-rt_n)^{-1/2}.
\]
In particular, taking $r=t_n$, we obtain
\[
\sum_{m\ge0}\overline{\hat{\phi}(m)}\hat{f}(m)(t_n)^m
\ge c_nB(t_n^2)^2\phi_0(t_n^2)
\quad(n\ge n_0).
\]
By Lemma~\ref{L:Blaschke},   $\inf_{n\ge n_0}B(t_n)^2>0$,
and from \eqref{E:infcond} we have $\sup_{n\ge n_0}c_n\phi_0(t_n^2)=\infty$.
Hence
\[
\sup_{n\ge n_0}\sum_{m\ge0}\overline{\hat{\phi}(m)}\hat{f}(m)(t_n)^m=\infty.
\]
By Abel's theorem, it follows
that $\sum_{m\ge0}\overline{\hat{\phi}(m)}\hat{f}(m)$
diverges, as claimed.
\end{proof}

%%%%%%%%%%%%%%%%%%%%%%%%%%%%%%%

%%%%%%%%%%%%%%%%%%%%%%%%%%%%%%%

\section{A limit formula}\label{S:Tapprox}

The main result of this section is a limit form of the formula \eqref{E:formula}.
Though it is more complicated than \eqref{E:formula}, 
it has the virtue that it holds for all Pythagor\-ean pairs
$(b,a)$ and all $f\in\cH(b)$.

\begin{theorem}\label{T:Tapprox}
Let $(b,a)$ be a Pythagorean pair. 
For  $\epsilon>0$, let $a_\epsilon$ be the outer function
such that $|a_\epsilon|=\max\{|a|,\epsilon\}$
a.e.\ on $\TT$ and $a_\epsilon(0)>0$. Set $\phi_\epsilon:=b/a_\epsilon$. 
\begin{enumerate}[\normalfont(i)]
\item If $f\in \cH(b)$, then
\[
\|f\|_{\cH(b)}^2=\sum_{m\ge0}|\hat{f}(m)|^2+
\lim_{\epsilon\to0^+}\sum_{m\ge0}\Bigl|\sum_{n\ge0}\overline{\hat{\phi}_\epsilon(n)}\hat{f}(m+n)\Bigr|^2.
\]
\item If $f\in H^2$ but $f\notin \cH(b)$, then
\[
\lim_{\epsilon\to0^+}\sum_{m\ge0}\Bigl|\sum_{n\ge0}\overline{\hat{\phi}_\epsilon(n)}\hat{f}(m+n)\Bigr|^2=\infty.
\]
\end{enumerate}
\end{theorem}

For the proof, we use a lemma on coanalytic 
Toeplitz operators.

\begin{lemma}\label{L:Tapprox}
Let $(\psi_\ell)$ be a sequence in $H^\infty$.
The following statements are equivalent:
\begin{enumerate}[\normalfont(i)]
\item $\|T_{\overline{\psi}_\ell}h\|_{H^2}\to0$  for all $h\in H^2$;
\item $\sup_\ell\|\psi_\ell\|_{H^\infty}<\infty$ and $\psi_\ell\to0$ pointwise on $\DD$.
\end{enumerate}
\end{lemma}

\begin{proof}
(i)$\Rightarrow$(ii): Suppose that $\|T_{\overline{\psi}_\ell}h\|_{H^2}\to0$
for each $h\in H^2$. 

We first show that $\sup_\ell\|\psi_\ell\|_{H^\infty}<\infty$.
As $\sup_\ell\|T_{\overline{\psi}_\ell}h\|_{H^2}<\infty$ for each $h\in H^2$,
the Banach--Steinhaus theorem implies that
$\sup_\ell\|T_{\overline{\psi}_\ell}\|<\infty$.
According to the Brown--Halmos theorem \cite[p.95]{BH63}, 
we have $\|T_{\overline{\psi}_\ell}\|=\|\psi_\ell\|_{H^\infty}$ for each~$\ell$. 
Hence $\sup_\ell\|\psi_\ell\|_{H^\infty}<\infty$, as claimed.

Now we show that $\psi_\ell\to0$ pointwise in $\DD$.
Since $\sup_\ell\|\psi_\ell\|_{H^\infty}<\infty$,
the sequence $(\psi_\ell)$ forms a normal family on $\DD$.
Therefore it suffices to prove that $\hat{\psi}_\ell(k)\to0$ as $\ell\to\infty$ for each $k\ge0$.
Let $k\ge0$ and set $f_k(z):=z^k$.
Then, for each $\ell$, we have
\[
\|T_{\overline{\psi}_\ell}f_k\|_{H^2}^2
=\sum_{m\ge0}\Bigl|\sum_{n\ge0}\overline{\hat{\psi}_\ell(n)}\hat {f}_k(m+n)\Bigr|^2
=\sum_{j=0}^k|\hat{\psi}_\ell(j)|^2
\ge |\hat{\psi}_\ell(k)|^2.
\]
Since $\|T_{\overline{\psi}_\ell}f_k\|_{H^2}\to0$ as $\ell\to\infty$,
we get $\hat{\psi}_\ell(k)\to0$ as $\ell\to\infty$,
as desired.

(ii)$\Rightarrow$(i): Suppose that $\sup_\ell\|\psi_\ell\|_{H^\infty}<\infty$ 
and $\psi_\ell\to0$ pointwise on $\DD$.
Since $\sup_\ell\|\psi_\ell\|_{H^\infty}<\infty$,
the sequence $(\psi_\ell)$ forms a normal family on~$\DD$,
and since $\psi_\ell\to0$ pointwise on $\DD$, the
convergence is in fact locally uniform on~$\DD$.
Therefore $\psi_\ell^{(j)}(0)\to0$ as $\ell\to\infty$ for each $j$,
and so $\hat{\psi}_\ell(j)\to0$ as $\ell\to\infty$.
Writing $f_k(z):=z^k$, it follows that
\[
\|T_{\overline{\psi}_\ell}f_k\|_{H^2}^2
=\sum_{j=0}^k|\hat{\psi}_\ell(j)|^2
\to0 
\quad(\ell\to\infty).
\]
Re-using the hypothesis that $\sup_\ell\|\psi_\ell\|_{H^\infty}<\infty$, 
we have $\sup_\ell\|T_{\overline{\psi}_\ell}\|<\infty$.
As the set $\{f_k:k\ge0\}$ spans a dense subspace of $H^2$,
it follows that $\|T_{\overline{\psi}_\ell}h\|_{H^2}\to0$
for all $h\in H^2$. 
\end{proof}

\begin{proof}[Proof of Theorem~\ref{T:Tapprox}]
(i) Let $f\in\cH(b)$. By Theorem~\ref{T:TbTa},
there exists a unique $g\in H^2$ such that 
$T_{\overline{b}}f=T_{\overline{a}}g$, and then
$\|f\|_{\cH(b)}^2=\|f\|_{H^2}^2+\|g\|_{H^2}^2$.
Also, for each $\epsilon>0$, we have
\[
\sum_{m\ge0}\Bigl|\sum_{n\ge0}\overline{\hat{\phi}_\epsilon(n)}\hat{f}(m+n)\Bigr|^2
=\|T_{\overline{\phi}_\epsilon}f\|_{H^2}^2.
\]
Thus, to prove (i), it suffices to show that
$T_{\overline{\phi}_\epsilon}f\to g$ in $H^2$ as $\epsilon\to0^+$.

Let $(\epsilon_j)$ be any positive sequence such that $\epsilon_j\to0$. For each $j$, set $\psi_j:=a/a_{\epsilon_j}$.
Then $\|\psi_j\|_{H^\infty}\le 1$ for all $j$ and $\psi_j\to1$
pointwise on $\DD$. By Lemma~\ref{L:Tapprox},
applied to the sequence $(\psi_j-1)$, we have $\|T_{\overline{\psi}_j}h- h\|_{H^2}\to0$  for each $h\in H^2$. In particular 
$\|T_{\overline{\psi}_j}g- g\|_{H^2}\to0$. On the other hand, 
since $\psi_j=a/a_{\epsilon_j}$ and $T_{\overline{a}}g=T_{\overline{b}}f$, we have
\[
T_{\overline{\psi}_j}g
=T_{1/\overline{a}_{\epsilon_j}}T_{\overline{a}}g
=T_{1/\overline{a}_{\epsilon_j}}T_{\overline{b}}f
=T_{\overline{\phi}_{\epsilon_j}}f.
\]
Thus $\|T_{\overline{\phi}_{\epsilon_j}}f-g\|_{H^2}\to0$.
As $(\epsilon_j)$ is an arbitrary positive sequence tending to zero,
we conclude that $\|T_{\overline{\phi_\epsilon}}f- g\|_{H^2}\to0$
as $\epsilon\to0^+$, as desired.

(ii) Let $f\in H^2$, and  suppose that
\begin{equation}\label{E:liminf}
\liminf_{\epsilon\to0^+}\sum_{m\ge0}\Bigl|\sum_{n\ge0}\overline{\hat{\phi}_\epsilon(n)}\hat{f}(m+n)\Bigr|^2<\infty.
\end{equation}
We shall show that $f\in\cH(b)$.
The inequality \eqref{E:liminf} amounts to saying that we have
$\liminf_{\epsilon\to0^+}\|T_{\overline{\phi}_\epsilon}f\|_{H^2}<\infty$.
Therefore there exists a positive sequence $(\epsilon_j)$ tending to zero
such that $\sup_j\|T_{\overline{\phi}_{\epsilon_j}}f\|_{H^2}<\infty$.
As bounded subsets of $H^2$ are weakly relatively compact,
we may replace $(\epsilon_j)$ by a subsequence and suppose that
$T_{\overline{\phi}_{\epsilon_j}}f$ converges weakly
 to some $g\in H^2$. It follows that 
 \begin{equation}\label{E:weak}
 T_{\overline{a}}T_{\overline{\phi}_{\epsilon_j}}f\to T_{\overline{a}}g
 \quad\text{weakly}.
 \end{equation}
 Set $\psi_j:=a/a_{\epsilon_j}$.
 As in the proof of part~(i), we have 
 $\|T_{\overline{\psi}_j}h-h\|_{H^2}\to0$ for each $h\in H^2$. 
In particular, 
\begin{equation}\label{E:innorm}
T_{\overline{\psi}_j}T_{\overline{b}}f\to T_{\overline{b}}f 
\quad\text{in norm}.
\end{equation}
Also, since $b\psi_j=a\phi_{\epsilon_j}$ for all $j$, we  have
\begin{equation}\label{E:thirdly}
T_{\overline{\psi}_j}T_{\overline{b}}f=T_{\overline{a}}T_{\overline{\phi}_{\epsilon_j}}f.
\end{equation}
Combining \eqref{E:weak}, \eqref{E:innorm} and \eqref{E:thirdly}, 
we deduce that $T_{\overline{b}}f=T_{\overline{a}}g$.
Consequently, by Theorem~\ref{T:TbTa}, we have $f\in\cH(b)$.
This completes the proof.
\end{proof}

I am indebted to the referee for pointing out the following
natural interpretation of Theorem~\ref{T:Tapprox} in terms of unbounded Toeplitz operators, in the
spirit of Sarason's paper \cite{Sa08}.

\begin{corollary}\label{C:Tapprox}
Let $(b,a)$ be a Pythagorean pair, let $\phi=b/a$
and,  for $\epsilon>0$, define $\phi_\epsilon$
as in Theorem~\ref{T:Tapprox}.
Let $T_{\overline{\phi}}$ be the unbounded operator
on $H^2$ defined by specifying that
\begin{align*}
\cD(T_{\overline{\phi}})&:=\{f\in H^2: \liminf_{\epsilon\to0^+}\|T_{\overline{\phi}_\epsilon}f\|_{H^2}<\infty\},\\
T_{\overline{\phi}}f&:=\lim_{\epsilon\to0^+}T_{\overline{\phi}_\epsilon}f
\quad(f\in\cD(T_{\overline{\phi}})).
\end{align*}
Then $T_{\overline{\phi}}$ is well-defined, $\cD(T_{\overline{\phi}})=\cH(b)$, and
\begin{equation}\label{E:Tphibar}
\|f\|_{\cH(b)}^2=\|f\|_{H^2}^2+\|T_{\overline{\phi}}f\|_{H^2}^2
\quad(f\in \cH(b)).
\end{equation}
\end{corollary}

\begin{proof}
If  $f\in\cD(T_{\overline{\phi}})$,
then $f\in\cH(b)$ by part~(ii) of Theorem~\ref{T:Tapprox},
and  by part~(i),
$\lim_{\epsilon\to0^+}\|T_{\overline{\phi}_{\epsilon}}f- g\|_{H^2}=0$,
where $g\in H^2$ satisfies $T_{\overline {b}}f=T_{\overline{a}} g$.
Thus $T_{\overline{\phi}}$
is well-defined. The remaining statements in the corollary
are immediate consequences of Theorem~\ref{T:Tapprox}.
\end{proof}

We can relate Corollary~\ref{C:Tapprox} to the
approach taken by Sarason in \cite{Sa08} as follows.
Once again, let $(b,a)$ be a Pythagorean pair
and let $\phi:=b/a$. Sarason defined an unbounded
analytic Toeplitz operator $T_\phi$ on $H^2$ by setting
\begin{align*}
\cD(T_\phi)&:=\{f\in H^2: \phi f\in H^2\},\\
T_{\phi}f&:=\phi f
\quad(f\in\cD(T_\phi)).
\end{align*}
He showed that $T_\phi$ is densely defined, and that its adjoint
$T_\phi^*$ satisfies $\cD(T_\phi^*)=\cH(b)$ and
\begin{equation}\label{E:Tphistar}
\|f\|_{\cH(b)}^2=\|f\|_{H^2}^2+\|T_{\phi}^*f\|_{H^2}^2
\quad(f\in \cH(b)).
\end{equation}
Comparing this with Corollary~\ref{C:Tapprox} leads us
to guess that, as unbounded operators on $H^2$,
\[
T_\phi^*=T_{\overline{\phi}}.
\]
This is indeed the case. We already know that
$\cD(T_\phi^*)=\cH(b)=\cD(T_{\overline{\phi}})$.
Also, comparing \eqref{E:Tphibar} and \eqref{E:Tphistar}
and using the polarization identity, we obtain
\[
\langle T_{\overline{\phi}}f,\,T_{\overline{\phi}}g\rangle_{H^2}=
\langle T_\phi^*f,\,T_\phi^*g\rangle_{H^2}
\quad(f,g\in\cH(b)).
\]
Applying this identity with $g=k_w$ for $w\in \DD$, 
and using the easily established identities 
$T_{\overline{\phi}}k_w=\overline{\phi(w)}k_w=T_\phi^*k_w$,
we deduce that 
\[
\langle T_{\overline{\phi}}f,\,\overline{\phi(w)}k_w\rangle_{H^2}
=\langle T_\phi^* f,\,\overline{\phi(w)}k_w\rangle_{H^2}
\quad(w\in\DD).
\]
It follows that $T_{\overline{\phi}}f=T_\phi^*f$
on $\DD\setminus\phi^{-1}(0)$, and hence on all of $\DD$,
as claimed.

%%%%%%%%%%%%%%%%%%%%%%%%%%%%%%%

\section{An open problem}

We conclude this article with a brief discussion of 
the following problem, which remains unsolved.

\begin{problem}\label{Pb:open}
Let $(b,a)$ be a Pythagorean pair and let $\phi=b/a$.
Is it true that, if $f\in H^2$ and
\begin{equation}\label{E:opencond}
\sum_{m\ge0}\Bigl|\sum_{n\ge0}\overline{\hat{\phi}(n)}\hat{f}(m+n)\Bigr|^2<\infty,
\end{equation}
then $f\in\cH(b)$ and
\begin{equation}\label{E:openformula}
\|f\|_{\cH(b)}^2=\sum_{m\ge0}|\hat{f}(m)|^2+\sum_{m\ge0}\Bigl|\sum_{n\ge0}\overline{\hat{\phi}(n)}\hat{f}(m+n)\Bigr|^2?
\end{equation}
\end{problem}

We have seen that the answer is affirmative if $\phi\in H^2$ 
(Theorem~\ref{T:H2}) or if $\phi$ is rational (Theorem~\ref{T:rational}), but for general $\phi\in N^+$,
the question remains open. We suspect that the answer 
is negative, but at present there is no counterexample.
(The counterexamples in this article, and 
those of which we are aware elsewhere in 
the literature, treat only the converse question, 
namely whether   \eqref{E:opencond}
holds for all $f\in\cH(b)$.)

The following proposition suggests a possible route towards
a counter\-example of the type that we seek.

\begin{proposition}\label{P:open}
Let $(b,a)$ be a Pythagorean pair and let $\phi=b/a$.
Suppose, in addition, that $b$
is an outer function (or, equivalently, that $\phi$ is outer). If there exists $f\in H^2\setminus\{0\}$
such that
\begin{equation}\label{E:openzerocond}
\sum_{n\ge0}\overline{\hat{\phi}(n)}\hat{f}(m+n)=0 \quad(m\ge0),
\end{equation}
then the answer to Problem~\ref{Pb:open} is negative
for this pair $(b,a)$.
\end{proposition}

\begin{proof}
Suppose, if possible, that the answer to 
Problem~\ref{Pb:open} is affirmative
for  $(b,a)$.
Then the condition \eqref{E:openzerocond} implies that 
$f\in\cH(b)$
and $\|f\|_{\cH(b)}^2=\|f\|_{H^2}^2$.
By Theorem~\ref{T:TbTa}, since $f\in\cH(b)$,
there exists $g\in H^2$ such  that 
$T_{\overline{b}}f=T_{\overline{a}}g$ and 
$\|f\|_{\cH(b)}^2=\|f\|_{H^2}^2+\|g\|_{H^2}^2$.
Comparing the two expressions for $\|f\|_{\cH(b)}^2$,
we see that $\|g\|_{H^2}^2=0$, so $g=0$, and hence 
$T_{\overline{b}}f=T_{\overline{a}}0=0$.
This in turn implies that $f\in (bH^2)^\perp$.
By assumption, $b$ is outer, so $bH^2$ is dense in $H^2$
and $(bH^2)^\perp=\{0\}$. Hence, finally, $f=0$,
contrary to assumption. This contradiction establishes
the proposition.
\end{proof}

%%%%%%%%%%%%%%%%%%%%%%%%%%%%%%%

\bibliographystyle{plain}
\bibliography{bibfile.bib}

\end{document}